\documentclass{amsart}
\usepackage{amsmath,amsfonts,amsthm}

\usepackage[usenames,dvipsnames]{color}

\newtheorem{thm}{Theorem} 

  \newtheorem{prop}[thm]{Proposition}

\newtheorem{theorem}[thm]{Theorem}

\newtheorem{lemma}[thm]{Lemma}
\newtheorem{corollary}[thm]{Corollary}
\theoremstyle{remark}
\newtheorem{remark}[thm]{Remark}

\newcommand{\n}{\|\cdot\|}

\renewcommand{\epsilon}{\varepsilon}

\title[Moduli of convexity and smoothness of reflexive subspaces of L$^1$]{Moduli of convexity and smoothness of reflexive subspaces of L$^1$}
\author{S. Lajara}
\address{Departamento de Matem\'{a}ticas, Escuela de Ingenieros Industriales, Universidad de Castilla-La Mancha, Campus Universitario\\ 02071 Albacete, Spain}
\email{sebastian.lajara@uclm.es}
\author{A.J. Pallar\'{e}s}
\address{Departamento de Matem\'{a}ticas, Universidad de Murcia, Campus de Espinardo\\30100 Murcia, Spain}
\email{apall@um.es}
\author{S. Troyanski}
\address{Departamento de Matem\'{a}ticas, Universidad de Murcia, Campus de Espinardo\\30100 Murcia, Spain}
\email{stroya@um.es}
\thanks{The authors were partially supported by MTM2008-05396/MTM Fondos Feder and Fundaci\'on S\'eneca 088848/PI/08 CARM. S. Lajara was also supported by JCCM PEII11-0132-7661. S. Troyanski was also supported by the Institute of Mathematics and Informatics, Bulgarian Academy of Sciences, 
Bulgarian National Fund for Scientific Research contract DO 02-360/2008.}


\begin{document}
\begin{abstract}
We show that for any probability measure $\mu$ there exists an equivalent norm on the space  $L^1(\mu)$ 
whose restriction to each reflexive subspace is uniformly   smooth and uniformly convex, with modulus of convexity of power type 2. This renorming provides also an estimate for the corresponding modulus of smoothness of such subspaces.
\end{abstract}




\maketitle

\section{Introduction and main results}

Let $X$ be a real Banach space with norm $\n$ and ·$S_X=\{x\in X: \|x\|=1\} $ its unit sphere. The moduli of convexity and smoothness of $X$ are the functions defined respectively by the formulas
$$\delta_X(\epsilon)= \inf\left\{ 1-\left\| \frac{x+y}{2} \right\| : x,y \in S_X,  
\|x-y\|=\epsilon\right\}, 0<\epsilon\leq 2,$$
and
$$\rho_X(\tau)= \sup\left\{\frac{\|x+\tau y\|+\|x-\tau y\|}{2}-1: x,y \in X 
\right\}, \tau >0.$$
The space $X$ is said to be {\em uniformly convex} if $\delta_X(\epsilon)>0$ for every $\epsilon>0$.  If, in addition, $\delta_X(\epsilon)\geq C \epsilon^q$, for some constant $C>0$ and $q\geq2$, we say that $X$ has modulus of convexity  of power type $q$.

The space $X$ is said to be {\em uniformly   smooth}
if $\lim_{\tau\to 0}\frac{\rho_X(\tau)}{\tau}=0$. 
If there exist constants $C>0$ and $1<p\leq 2$ such that $\rho_X(\tau)\leq C\tau^p$, we say that $X$ has modulus of smoothness of power type $p$.

It is well-known (see e.g. \cite[II, p. 63]{LZ} or \cite[Chapter 9.1]{FHHM}) that for $p>1$ the canonical norm on $L^p$ is uniformly convex and uniformly   smooth. Moreover, if $1<p\leq 2$ 
then $L^p$ has modulus of convexity of power type $2$ 
and modulus of smoothness of power type $p$. 

In \cite{Rosen73}, H.P. Rosenthal studies the subspaces  of $L^p(\mu)$, for a probability measure $\mu$, and shows that every reflexive subspace $X\subset L^1(\mu)$ embeds  in (is linearly homeomorphic to a subspace of)  $L^p(\nu)$ for some $p$, $1<p\leq 2$, and some probability measure $\nu$ such that $d\nu=\phi_X\, d\mu$ for some positive measurable function $\phi_X$. 
In particular,  $X$ admits an equivalent norm  with modulus of convexity of power type 2, and  modulus of smoothness of type $p$, for some $1<p\leq 2$. The renorming in this assertion depends naturally on the specific subspace $X$. In \cite{bretDC,Lav} it was shown that the class of reflexive subspaces of $L^1(\mu)$ is quite big. For example reflexive subspaces of Orlicz function spaces 
$L_M([0,1])$ with natural requirements  on 
the Orlicz function $M$ embed in $L^1([0,1])$.

The space $L^1(\mu)$ admits an equivalent norm, namely an Orlicz norm, whose restriction to every reflexive subspace is uniformly convex \cite{DKT}. 
An analogous statement for uniform smoothness was proved in \cite{FMZ} using some transfer techniques. Both results provide us a weaker version of Rosenthal theorem that every reflexive subspace of $L^1(\mu)$ is superreflexive (i.e. admits equivalent norms that are uniformly convex and/or uniformly   smooth). However, these results do not give
any 
 information 
 about the asymptotic behaviour at 0 of 
 the moduli of convexity and smoothness. 

The aim of this note is to construct an equivalent norm on $L^1(\mu)$ whose restriction to every reflexive subspace yields quantitative estimates of the moduli 
by means of the following indexes defined in terms of the distributions of  functions, $F_f(t)=\mu(\{|f|>t\})$, $f\in L^1(\mu)$. Namely, for a subspace $X$ of $L^1(\mu)$ we consider
\begin{align}
C_X(t)&=\inf \{ F_f(t): f\in X, \|f\|_1=1\}, 0<t<1, \text{ and} \label{eq:C-X}\\
G_X(t)&=\sup \left\{ \int_t^{+\infty} F_f(u)\, du: f\in X, \|f\|_1=1\right\}, t>0. \label{eq:G_X}
\end{align}

In the rest of this section we formulate our results   leaving the proofs for the next sections.
Our main theorem is the following

\begin{theorem} \label{maintheorem} 
Let  $\mu$ be a probability measure. Then there exist an equivalent norm $\n$ on $L^1(\mu)$ and positive constants $K_1$ and $K_2$  such that for
the moduli of convexity and smoothness of  the norm $\n$ 
in every subspace $X$  of  $ L^1(\mu)$, 
the inequalities hold
\begin{equation}\tag{A} \delta_X(\varepsilon)\geq K_X \epsilon^2, \text{ for } 0< \varepsilon\leq 2, \text{ where } K_X= K_1 \sup_{0<t<1}\{t^2C_{X}^3(t)
\},
\end{equation} 
 and 
\begin{equation}\tag{B}
\rho_X(\tau)\leq K_2 \tau^2  \int_0^{1/\tau} G_X(t) dt \tau>0.
\end{equation}
\end{theorem}

Note that the estimate (A) is meaningful  only if $C_X(t)>0$ for some $t>0$ and in this case it implies that the restriction to $X$ of the new norm is uniformly convex with modulus of convexity of power type 2, and consequently $X$ is superreflexive. 

We also observe that the estimate (B) would provide some useful information only if 
$\lim_{t\to+\infty} G_X(t)=0$ ($G_X(t)$ is non-increasing). For 
better  understanding of the
function $G_X(t)$, note  that using  Rieman-Stieltjes integrals and integration by parts, we have
$$
\int_{\{| f |>t\}} |f| \,d\mu=-\int_t^{\infty} x \,dF_f(x) =tF_f(t)+\int_t^{+\infty} F_f(u)du\geq \int_t^{+\infty} F_f(u)du.$$%
Thus, for all $t>0$, 
\begin{equation}\label{GXt:1}1\geq \sup  \left\{ \int_{\{| f |>t\}} |f| \,d\mu: f\in X, \|f\|_1=1 \right\}\geq G_X(t)
.\end{equation}
 Now it is not difficult to verify that $\lim_{t\to+\infty} G_X(t)=0$ implies that $$\lim_{\tau\to0} \tau\int_{0}^{1/\tau} G_X(t)dt=0,$$ and then (B)  in turn implies that the restriction to $X$ of the norm in the theorem is uniformly   smooth, and that $X$ is superreflexive.

For the converse, when $X$ is reflexive, its unit sphere must be a relatively weakly compact subset of $L^1(\mu)$ and after the 
characterization due to Dunford and Pettis (see e.g. \cite[5.2.9]{AlbKal}) must be {\em equi-integrable.
} Recall that a bounded subset $W\subset L^1(\mu)$ is called equi-integrable (or uniformly integrable) if
\begin{equation}\label{eq:uniformintegrable1}
\lim_{\mu(A)\to 0} \sup\left\{\int_A |f| \, d\mu: f\in W \right\} =0,
\end{equation}
which is equivalent to
\begin{equation}\label{eq:uniformintegrable2}
\lim_{t\to+\infty} \sup  \left\{ \int_{\{| f |>t\}} |f| \,d\mu: f\in W \right\}=0.
\end{equation}

Clearly, from \eqref{GXt:1} and \eqref{eq:uniformintegrable2} we can apply (B) as in the above argument to get that $X$ is uniformly smooth renormable. 

Also from the equi-integrability of the sphere $S_X$ we deduce that $C_X(t)>0$ for every $t\in(0,1)$ and we can apply (A) to obtain that $X$ is uniformly convexifiable
. Indeed, assume the contrary, i.e there is  $t\in(0,1)$ with $C_X(t)=0$, and  find a sequence $f_n\in S_X$ such that $F_{f_n}(t)\to 0$. Now, for $A_n=\{|f_n|>t\}$  we have  $\mu(A_n)\to 0$ and   from \eqref{eq:uniformintegrable1} we get $\|f_n\chi_{A_n}\|_1=\int_{A_n}|f_n|d\mu < \epsilon$ for  $\epsilon=1-t  >0$ and $n$ big enough. Let us denote by 
$B_n=\{|f_n|\leq t\}$ the complement of $A_n$, and observe that  we get the contradiction:
$$t \geq t\mu(B_n)\geq \|f\chi_{B_n}\|_1=1-\|f\chi_{A_n}\|_1 >1-\epsilon= t.$$

After these remarks, we are ready to set the following characterization

\begin{corollary}\label{corolario}
Let $\mu$ be a probability measure. For a closed subspace $X\subset L^1(\mu)$, the following are equivalent
\begin{enumerate}\renewcommand{\labelenumi}{(\roman{enumi})}
\item $X$ is superreflexive;
\item $X$ is reflexive;
\item $C_X(t)>0$ for some $t\in (0,1)$;
\item $\lim_{t\to+\infty} G_X(t)=0$. 
\end{enumerate}
\end{corollary}

When we have good estimations for $G_X(t)$,  we can apply (B) of Theorem \ref{maintheorem} in order to approach the power type for the modulus of smoothness. 

For a first example, a direct application of the theorem gives us 
\begin{corollary}  Let  $X\subset L^1(\mu)$ be a closed subspace such that   $G_X(t)$ is integrable in $[0,+\infty)$.    
 Then
 $$    \rho_X(\tau)   \leq K \,  \tau^2, $$  where $K=K_2 \int_0^{+\infty} G_X(t)dt <+\infty$. In this case, $X$ with the norm from Theorem \ref{maintheorem} has power type 2 for both  moduli of convexity and smoothness. 
 \end{corollary}
For instance, if $R$ 
is the space generated by the Rademacher functions in $L^1([0,1])$, 
there is  a well known upper bound for its distribution function (see, for example, \cite{MoSm}). Using this estimate and  
the Khintchine inequality we get
$$G_R(t)=\sup\left\{\int_t^{+\infty}F_g(x)dx: g\in R, \|g\|_1=1\right\}\leq \int_t^{+\infty} e^{-\frac{1}{2} c_1^2 x^2}dx\leq c_2 e^{-\frac{1}{2} c_1^2 t^2}$$
for some constants $c_1$ and $c_2$. Thus, $G_R(t)$ is integrable
 and  $R$ is 2-uniformly smooth and 2-uniformly convex.

For another applications, observe that working with the representation as Riemann-Stieltjes integrals it is not difficult to prove that  when $X$ is in $L^p(\mu)\subset L^1(\mu)$, $p>1$, and the norms $\n_p$ and $\n_1$ are equivalent on $X$, then 
$$ t^{p-1} G_X(t)\leq \frac{1}{p} \sup\left\{\int_{\{|g|>t\}}|g|^pd\mu: g\in X, \|g\|_1=1\right\},$$ 
and  $G_X(t)=o({1}/{t^{p-1}})$.  

Let us assume that $G_X(t)=O (1/t^{p-1})$. 
In this case, for  $p>2$,  $G_X(t)$ is integrable on $[0,+\infty)$  and then the modulus of smoothnes is of type 2. For the case $1<p<2$ it is easy to prove that $\int_0^{\frac 1\tau} G_X(t)\leq c+k \tau^{p-2}$ for some constants $c$ and $k$, and bringing this inequality to (B) in Theorem \ref{maintheorem}	
 we obtain the power type $p$ for the modulus of smoothness. Finally, the case $p=2$ is similar to the above, but now  taking primitives we obtain that $\int_0^{\frac 1\tau} G_X(t)\leq k |\log \tau|$ for some constant $k$. Thus

\begin{corollary}\label{smooth:type} Let $\mu$ be a probability measure and  $X\subset L^1(\mu)$ be a subspace such that $G_X(t)=O (1/t^{p-1})$ with $p>1$. Then 
\begin{enumerate} \item 
$\rho_X(\tau)=O( \tau^2)$ if $2<p$,
\item 
$\rho_X(\tau)=O( \tau^{p})$ if $1<p<2$,
 and  \item $\rho_X(\tau)=O( \tau^2 |\log \tau|)$ if $p=2$.
\end{enumerate}
\end{corollary}

Let us note that  $G_X(t)=O (1/t^{p-1})$ implies that each function of $X$ is in $L^{p'}(\mu)$ for $p'<p$ and that the corollary can be applied to the subspaces $X$ in $L^p(\mu)\subset L^1(\mu)$, $p>1$, where the norms $\n_p$ and $\n_1$ are equivalent, giving that the norm from Theorem \ref{maintheorem} is $\min\{p,2\}$ uniformly smooth except in the case $p=2$. However  for this case some facts of the proof of the main theorem show that the new norm is 2 uniformly smooth (see Remark \ref{remark1}).

In general, the estimate (B) in Theorem \ref{maintheorem} does not give us the power type behaviour of $\rho_X$. 
 To get examples where this happens, we focus our attention on reflexive subspaces $E_f$ of $L^1([0,1])$ that can be generated using the Rademacher functions and different positive densities (weights) $f\in L^1([0,1] )$.

Note that   
in \cite{veraar}  a weighted version of the Khintchine inequalities was proved, showing that for a  strictly positive weight $f \in  L^p([0,1])$, $p>1$, the space $E_f$ defined as the closed span in $\|.\|_1$ of the set $\{f r_n: n\in \mathbb{N}\}$, is a copy of $\ell_2$. In the next proposition we see that for any positive weight $f\in L^1([0,1])$, and in particular if $\int_0^1 f(x)^p dx=+\infty$ for every $p>1$, it is possible to find copies of $\ell_2$ in $L^1([0,1])$ defined in a similar way.

\begin{prop}\label{subespacioR}
Let $f\in L^1[0,1]$, $f\geq 0$ and $\|f\|_1=1$, and consider the sequence of Rademacher functions $r_n(x)=\mathop{sign}(\sin(2^n\pi x))$. Then, there is a subsequence $r_{n_k}$ such that the space $E_f$ generated by the sequence $(f r_{n_k})_k$ is isomorphic to the reflexive space $R$ generated by the Rademacher sequence $(r_k)_k$, and through the Khintchine inequalities, is a copy of $\ell_2$.
\end{prop}

We postpone the proof for the last section where we also find a function $f$ in $ L^1([0,1])$,
with $\int_0^1|f(x)|^pdx=+\infty$ if $p>1$ and such that
 $G_{E_f}(t)\geq 1/({4\log (4 t)})$ 
 if 
 $t$ is big enough (Proposition \ref{funLog}). For such a function $f$, $G_{E_f}(t)\neq O(1/t^{p-1})$ for any $p>1$ and Corollary \ref{smooth:type} does not apply.

This example shows us that the estimate (B) in Theorem \ref{maintheorem} for $\rho_X$ is not sharp in general under renorming $X$.  However the inequality  (A) for $\delta_X$ allows us to find some non trivial estimate for the type of $X$ and to renorm the subspace $X$ in such a way to get a power type estimate 
for the modulus of smoothness of the new norm in terms of $C_X(t)$. In order to do this we can use the results of \cite{IPT,IT}  where it was shown that if $\delta_X(\epsilon)\geq k\epsilon^2$, then we can renorm $X$ to obtain that the new norm has modulus of smoothness of power type $p=p(k)>1$ and this estimate is asymptotically sharp when $k$ goes to 1/8 or 0 (remember that $\delta_X(\epsilon) \leq \delta_{L^2}(\epsilon)=(1/8) \epsilon^2 + o(\epsilon^2)$ \cite[II,pag. 63]{LZ}).

\section{Proof of the Main Theorem}

Let $(\Omega, \Sigma, \mu)$ be a probability measure space. 
The equivalent norm on $L^1(\mu)$ that verifies the thesis of Theorem \ref{maintheorem} will be the Luxemburg norm in the Orlicz space $L^1(\mu)$ for a suitable Orlicz function $M(t)$ (see e.g \cite{Kras,Rao} for references). 

\subsection*{The Orlicz function space}

Consider the function $\varphi(t)= 2$  if $0\leq t\leq 1$ and $\varphi(t)= 8/{(1+t)^2}$  if $t>1$, 
and let $M(t)$ be the function defined as a second primitive of $\varphi$ by the expression
$M(t)=\int_0^{|t|} \varphi(u)(t-u) \,du .$ 
It is clear that $M$ is an Orlicz function, i.e., an even, continuous, convex, increasing in $[0,+\infty)$ and $M(0)=0$. It verifies that 
$M(t)=t^2$ if $|t|\leq 1$, $M'(t)$ is concave in $[0,+\infty)$, $M''(t)=\varphi(t)$, and that 
$$\lim_{t\to+\infty}\frac{M(t)}{t}=\lim_{t\to+\infty}M'(t)=\int_0^{+\infty} \varphi(u)\, du<+\infty.$$ 
This identity shows (see e.g \cite[II.13.7]{Kras}) that  the Luxemburg norm associated to $M$ defined as
$\|f\|:= \inf \left\{ \lambda >0 : \int_{\Omega} M(f/\lambda)\, d\mu \leq 1  \right\}$  is equivalent to $\n_1$ in $L^1(\mu)$, and,   
because $M$ is normalized ($M(1)=1$), there is some constant $k>0$ such that 
\begin{equation}\label{eq:equiv-norms}
k \|.\|\leq 	\|.\|_1 \leq \|.\|
\end{equation}


Since $M'(t)$ is concave on $[0,+\infty)$ and $M'(0)=0$ we have   $\alpha M'(u)\geq M'(\alpha u)$ for $\alpha \geq 1$ and $u> 0$. Thus, for $t\geq 0$,
\begin{equation}\label{eq:alpha2}
\alpha^2 M(t) =\int_0^t \alpha^2 M'(u) du \ge \int_0^{t} \alpha M'(\alpha u) du = \int_0^{\alpha t} M'(s) ds = M(\alpha t).
\end{equation}
In particular we have  $4M(t)\geq M(2t)$ (that gives the $\Delta_2$ condition for $M$ \cite{Kras,Rao}) and 
$$3M(t)\geq M(2t)-M(t)=\int_t^{2t} M'(u) du \geq t M'(t) = t\int_0^t M''(u) du\geq t^2 M''(t).$$
Thus
\begin{equation}
 \label{lem:ineq1}
 M(t) \geq \frac{1}{3} t^2 M''(t), \text{ for all } t\geq 0.
\end{equation}

Using the inequalities  \eqref{eq:alpha2} and \eqref{lem:ineq1} we shall to prove the following lemma that is the key for proving our main theorem.

\begin{lemma} \label{lem:ineq2} The function $M$ satisfies the inequalities
$$\frac{1}{4} M''(c)(a-b)^2\leq M(a)+M(b) - 2 M \left(\frac{a+b}{2}\right) 
\leq 16\, M\left(\frac{a-b}{2}\right)$$
for all $a, b\in \mathbb{R}$, and $c=\max\{|a|,|b|\}$.
\end{lemma}
\begin{proof}
Taylor's formula for $M$ gives
$$M(a) +M(b)- 2 M\left(\frac{a+b}{2}\right) 
 = \frac{(a-b)^2}{4} M''\left(\frac{a+b}{2}+\theta\frac{a-b}{2}\right) \label{taylor2g}$$
with $\theta \in(-1,+1)$.
As $M$ is an even function and $M''$ is decreasing in $[0,+\infty)$, it is easy to get that $M''(\tfrac{a+b}{2}+\theta\tfrac{a-b}{2})\geq M''(c)$, which provides the left inequality.

For the right inequality we assume that $0\leq|a|\leq|b|$, $a\neq b$, and consider separately three cases:

\noindent (a)  $0< \frac{b}{2}\leq a\leq b$ $(c=b)$. Put $d=\frac{a+b}{2}+\theta 	\frac{a-b}{2}	\in[\frac{b}{2},b]$. By inequality  \eqref{lem:ineq1} we get
$$M''(d)\leq 3 \frac{M(d)}{d^2}\leq 3 \frac{M(b)}{\frac{1}{4}b^2}=12 \frac{M(c)}{c^2}.$$
Going back to Taylor's formula from above we obtain
$$M(a) +M(b)- 2 M\left(\frac{a+b}{2}\right) \leq 3 (a-b)^2 \frac{M(c)	}{c^2}.$$

\noindent (b) $0\leq a< \frac{b}{2}$ $(c=b)$. In this case $|a-b|\geq \frac{b}{2}$,  and we have $$M(a) +M(b)- 2 M\left(\frac{a+b}{2}\right) \leq 
M(b)=4  \left(\frac{b}{2}\right)^2\frac{M(c)}{c^2}\leq 4 (a-b)^2\frac{M(c)}{c^2}.$$

\noindent (c) $0\leq -a \leq b$ or $0\leq a \leq -b$ $(c=b)$. In both cases $|a-b|\geq |b|$ and
$$M(a) +M(b)- 2 M\left(\frac{a+b}{2}\right) \leq 
2 M(b)=2  b^2\frac{M(c)}{c^2}\leq 2 (a-b)^2\frac{M(c)}{c^2}.$$

To complete the chain of inequalities we use  \eqref{eq:alpha2} with $\alpha= \frac{2 c}{|a-b|}>1$ and we get
$$ 4(a-b)^2\frac{M(c)}{c^2} =16 \, \alpha^{-2}M\left(\alpha\, \frac{a-b}{2}\right)\leq 16\,  M\left(\frac{a-b}{2}\right).$$
\end{proof}

The upper bound we propose for 
$M(a)+M(b)-2M(\frac{a+b}{2})$ can be deduced from
\cite[Lemma 4]{MaTr78}, in fact we only improve the constant. We note 
that in \cite{MaTr78} an estimate  similar to our estimate from below  is obtained assuming  $ tM'(t)/M(t)\geq p>1$ for $t>0$. Clearly for our function  $\lim_{t\to+\infty}tM'(t)/M(t)=1$  and 
the result from \cite{MaTr78} is not appliable. For this reason we involve the second derivative of $M$. 

\subsection*{(A) Estimate for the modulus of convexity.}

To get the estimation (A) of Theorem \ref{maintheorem} we prove the following more general statement relative to the Luxemburg norm for our Orlicz function. 

\begin{lemma}\label{prop:generaluc} For every $u,v \in L^1(\mu)$, $\|u\pm v\|=1$, and every $t>0$ the inequality holds
\begin{equation}\label{eq:generaluc}
1-\|u\| \geq \left(\frac{k}{9}\right)^2 t^2 \mu(\{|v|> t\|v\|_1 \})^3 \|v\|^2,
\end{equation}
where $k$ is the isomorphic constant from \eqref{eq:equiv-norms}.
\end{lemma}
\begin{proof}
Without loss of generality we may assume that 
$\mu(\{|v|> t\|v\|_1\})>0$.

Recall that $k\|\phi\|\leq \|\phi\|_1\leq \|\phi\|$ for every $\phi$ in $L^1(\mu)$.

Set $z=v/\|v\|$, then we have
$$\mu(\{|z|> k t\})\geq \mu(\{|z|>t\|z\|_1\})=\mu(\{|v| > t\|v\|_1\})>0.$$
On the other hand the Tchebychev  inequality   for  $s= 4/\mu(\{|v| > t\|v\|_1\})$  and $\phi$ in $L^1(\mu)$ gives
$$\mu(\{ s \|\phi\| < |\phi|\}) \leq \mu(\{ s\|\phi\|_1 < |\phi|\})\leq \frac{1}{s}= \frac{\mu(\{|v| > t\|v\|_1\})}{4}.$$
Now, we can use the last inequalities for our function $z$  to get the estimate
\begin{align}
\mu(\{ s \geq |z| > k t\})&=\mu(\{ |z| > k t\})- \mu(\{ s < |z|\}) \notag \\ &\geq \mu(\{|v| > t\|v\|_1\}) - \frac{1}{s} = \frac{3}{4} \mu(\{|v| > t\|v\|_1\}).\label{musubindex}
\end{align}

Note that $\|u\|\leq 1$, and $\|v\|\leq 1$. Then for each  $\omega\in \Omega$ such that $|u(\omega)|\leq s\|u\|\leq s$  and $s\geq |z(\omega)|>k t$, we have $|v|(\omega)\leq s$ and therefore $|u(\omega)\pm v(\omega)|\leq 2s$. As $M''$ is decreasing in $[0,+\infty)$ by applying Lemma \ref{lem:ineq2} with $a=u(\omega)+v(\omega)$ and $b=u(\omega)-v(\omega)$ we get
\begin{equation}
M(u(\omega)+v(\omega)) +M(u(\omega)-v(\omega))- 2 M(u(\omega)) \geq  \|v\|^2 z^2(\omega) M''(2s) \label{taylor2-d}.
\end{equation}
	The $\Delta_2$ condition for $M$ gives that $ \int_\Omega M(f) \, d\mu=1 $,  if $\|f\|=1$ \cite[III.3.4.6]{Rao}. 
	By integrating \eqref{taylor2-d} over $\Omega$ we obtain
\begin{align}
2&\left(1 -\int_{\Omega} M(u)\, d\mu\right) \notag \\
&= \int_{\Omega}M(u+v)\, d\mu+\int_{\Omega} M(u-v)\, d\mu-2\int_{\Omega} M(u)\, d\mu \notag \\
&\geq  \|v\|^2 M''(2s) \int_{\{ s\geq |z|>kt\}\setminus \{ |u|>s\|u\| \}} z^2 \, d\mu \notag \\
&\geq \|v\|^2  M''(2s) k^2t^2 \left(\mu(\{ s\geq |z|>kt\})-\mu(\{ |u|>s\|v\| \})\right) \notag \\
&>\frac{1}{2} \|v\|^2 M''(2s) k^2 t^2 \mu(\{|v| > t\|v\|_1\}) = 4 \|v\|^2   	\frac{1}{(1+2s)^2}k^2 t^2 \mu(\{|v| > t\|v\|_1\}) \notag \\
&\geq  \|v\|^2 \left(\frac{2 k}{9}\right)^2 t^2 \mu(\{|v| > t\|v\|_1\})^3.\label{modconv0}
\end{align}

On the other hand, having in mind the inequality  \eqref{eq:alpha2} with $\alpha=\frac{1}{\|u\|}$ we get
$$\frac{1}{\|u\|^2}\int_{\Omega} M(u(\omega))\, d\mu(\omega) \geq \int_{\Omega} M\left(\frac{u(\omega)}{\|u\|}\right)\, d\mu(\omega)=1,$$
and using this inequality in \eqref{modconv0} we obtain the inequality we are looking for:
\begin{align}
 1-\left\|u\right\| &\geq \frac{1}{2} \left(1-\left\|u\right\|^2 \right)
 \geq \frac{1}{2}\left(1- \int_{\Omega} M(u(\omega))\, d\mu(\omega)\right) \notag \\ &
 \geq \left({k}/{9}\right)^2 t^2 \mu(\{|v| > t\|v\|_1\})^3 \|v\|^2. \notag
\end{align}
\end{proof}

Pick $f,g\in X$, $\|f\|=\|g\|=1$ with $\|f-g\|=\epsilon$. Setting  $u=(f+g)/2$ and $v=(f-g)/2$ in  Lemma \ref{prop:generaluc}, $K_1=({k}/{18})^2$, $K_X= K_1 \sup\{ t^2 C_X^3(t), 0<t<1 \}$ and  according to the  definition of the modulus of convexity, we get 
\begin{equation} \tag{A}
\delta_X(\epsilon)\geq K_X \epsilon^2,   \text{ for } 0<\epsilon\leq 2,
\end{equation}
as we wanted to show.
\begin{remark}
Lemma \ref{prop:generaluc} implies that $L^1(\mu)$ with the new Orlicz norm is {\itshape uniformly rotund in every direction}  which was proved in  \cite{DKT} essentially using a strictly convex Orlicz function  that satisfies the $\Delta_2$ condition. On the other hand, the derivative of our function $M$ is concave and  belongs to the class of functions considered in \cite{Maleev}. So, according to the main result of this work, that norm is uniformly G\^ateaux smooth. 
\end{remark}

\subsection*{(B) Estimate for the modulus of smoothness. }

The proof of (B)  is based on the right hand side  inequality in Lemma \ref{lem:ineq2}. 
In order to simplify the computations we need the following
\begin{prop}[Figiel \cite{Figiel76}] For every Banach space $X$ and $\tau>0$ we have
\begin{equation}\label{eq:modFigiel} \rho_X(\tau)\leq 16  \sup
\left\{\tfrac{\|x+\tau y\|+
\|x-\tau y\|}{2}, x,y \in S_X ,  x \perp y \right\} 
\end{equation}
where $x\perp y $ means that there is some $x^*\in S_X$ such that $x^*(x)=1$ and $x^*(y)=0$. 
\end{prop}

\begin{proof}[Proof of (B) in Theorem \ref{maintheorem}]

Let $X$ be  a subspace of $L^1(\mu)$ and  $\|.\|$ be  the Luxemburg norm associated to our Orlicz function $M$.
Since $M(\alpha t)\leq \alpha M(t)$ for every $\alpha\in [0,1]$ and $t>0$, 
for $f,g\in X $, $\|f\|=\|g\|=1$,  $f\perp g$, 
we have
$$1\leq \|f\pm\tau g\| \leq \int_\Omega M(f\pm\tau g) \, d\mu.$$
On the other hand, the right inequality in Lemma \ref{lem:ineq2} for $a=f(\omega)+\tau g(\omega)$ and $b=f(\omega)-\tau g(\omega)$ reads
$$M(f(\omega)+\tau g(\omega))+M(f(\omega)-\tau g(\omega))- 2 M(f(\omega))\leq 16 M(\tau g(\omega)),$$
and integrating over $\Omega$  we get
\begin{align}
\|f+\tau g\| +\|f-\tau g\| -2 &\leq \int_\Omega M(f+\tau g) \, d\mu +\int_\Omega  M(f-\tau g) \, d\mu - 2\int_\Omega  M(f) \, d\mu\notag \\
                                        &\leq 16 \int_\Omega  M(\tau g) \, d\mu . \label{eq:totheremark}
\end{align}
This, together with  inequality \eqref{eq:modFigiel} gives 
$$\rho_X(\tau)\leq 128 \sup\left\{ \int_\Omega M(\tau g) \, d\mu: \|g\|=1 , g\in X \right\}.$$

Having in mind \eqref{eq:equiv-norms}, and considering $g_1=g/\|g\|_1$ we have $$M(\tau g)\leq \|g\|_1M(\tau g_1)\leq M(\tau g_1).$$ Thus, we have also
$$
\rho_X(\tau)\leq 128 \sup\left\{ \int_\Omega M(\tau g) \, d\mu: \|g\|_1=1 , g\in X \right\}.
$$

Our Orlicz function $M(x)$ is equal to $x^2$ for $|x|\leq 1$ and equivalent to $x$ when $x\to+\infty$. Then, it is easy to find a constant $C>1$ such that  $M(x)\leq C |x|$ for $|x|\geq 1$. Thus for $K=128C>0$ we have
\begin{equation}\label{eq:us1}
\rho_X(\tau)\leq K \sup_{g\in X, \|g\|_1=1}\left\{ \tau^2 \int_{\{ |g|\leq \frac{1}{\tau}\}}|g|^2\, d\mu+\tau\int_{\{ |g|>\frac{1}{\tau}\}} |g| \, d\mu \right\}. 
\end{equation}

In order to put the sum in the brackets in the last inequality as in a single integral, we are going to express the two integrals as Riemann-Stieltjes integrals  
and after an integration by parts we will write them as double integrals.
\begin{align}
I_1&:=\tau^2\int_{\{ |g|\leq \frac{1}{\tau}\}}|g|^2\, d\mu=-\tau^2\int_{0}^{1/\tau}x^2dF_g(x) \notag \\
    &=-F_g(1/\tau)+2\tau^2\int_0^{1/\tau}xF_g(x)dx 
    =-F_g(1/\tau)+2\tau^2\int_0^{1/\tau}F_g(x)\int_0^x  dt \, dx\notag \\
   &=-F_g(1/\tau)+2\tau^2\int_0^{1/\tau}\int_{t}^{1/\tau} F_g(x) dx \,dt.\notag\\
I_2&:=\tau\int_{\{ |g|> \frac{1}{\tau}\}}|g|\, d\mu=-\tau\int_{1/\tau}^{+\infty}xdF_g(x)  	
	 =F_g(1/\tau)+\tau\int_{1/\tau}^{+\infty}F_g(x)dx \notag \\
     &=F_g(1/\tau)+\tau^2\int_0^{1/\tau}\int_{1/\tau}^{+\infty}F_g(x)dx \, dt
     \notag 
\end{align}

$$I_1+I_2 \leq 2\tau^2\int_0^{1/\tau}\int_{t}^{+\infty}F_g(x)dx \, dt. 
$$

This inequality and \eqref{eq:us1} imply (B) with $K_2=2 K$
\begin{equation} \tag{B}
\rho_X(\tau)\leq K_2 \tau^2\sup_{g\in X, \|g\|_1=1}\left\{\int_0^{1/\tau}\int_t^{+\infty} F_g(x) dx\right\}\leq K_2 \tau^2  \int_0^{1/\tau}G_X(t) dt.
\end{equation}
\end{proof}

\begin{remark}\label{remark1}
For every $1<p\leq 2$ there is some contant $c_p>0$ such that 
$M(u)\leq c_p |u|^p$ for all $u$. When $X$ is in $L^p(\mu)\subset L^1(\mu)$,  and the norms $\n_p$ and $\n_1$ are equivalent on $X$ ($\|g\|_p\leq C\|g\|$ for $g\in X$),  in \eqref{eq:totheremark} we have 
$$\|f+\tau g\| +\|f-\tau g\| -2 \leq \textstyle 16 \int_\Omega  M(\tau g)d\mu\leq 16c_p \tau^p \|g\|_p^p \leq  16c_p C^p \,\tau^p $$
for $f$ and $g\in X$, $\|f\|=\|g\|=1$, and $f\perp g$. Thus, the subspace $X$, endowed with the  norm from Theorem \ref{maintheorem} has modulus of convexity of power type  2 and modulus of smoothness of power type $p$.
\end{remark}

\section[Copies of $\ell_2$]{Some reflexive subspaces of $L^1([0,1])$ that are copies of $\ell_2$.}

\subsection*{Proof of Proposition \ref{subespacioR}.}

Let us introduce some notation before  proceeding with the proof.
Consider the dyadic tree $T=\{0,1\}^{[\mathbb{N}]}=\bigcup_{n=1}^{+\infty}\{0,1\}^n$ 
with the order defined by  $s\leq v$ if  $s=\{s_j\}_{1\leq j\leq n_s}$, $d=\{d_j\}_{1\leq j\leq n_d}$, $n_s\leq n_d$, and $s_j=d_j$ for $j\leq n_s$, i.e. $s$ 
 is a predecessor of 
$v$. 

We describe all the dyadic intervals indexed in $T$, $\{I_s^n: s\in \{0,1\}^n, n\in \mathbb{N} \}$. We   begin  with 
$I_{\{0\}}^1=[0,1/2) \hspace{10pt}\text{and}\hspace{10pt}I_{\{1\}}^1=[1/2,1),$
and for a  given  $s\in \{0,1\}^n$ if 
$I_{s}^{n}=[j(s)/2^n,(j(s)+1)/2^n)$ for some  $j(s)\in\{0,1,...,2^n -1\}$,  
we define 
$$I_{\{s,0\}}^{n+1}= [j(s)/2^{n},(2j(s)+1)/2^{n+1}),\text{ and  }$$ $$I_{\{s,1\}}^{n+1}=[(2j(s)+1)/2^{n+1},(j(s)+1)/2^{n}).$$

Observe that 
the Rademacher function $r_n$ is constant on each $I_s^n$ and takes the value $1$ or $-1$ depending 
on whether the last digit  $s_n$ in $s=\{s_1,...,s_n\}$,  is $0$ or $1$.

 We denote by $D_m$ ($m\in\mathbb{N}$) the family of all sets $A\subset [0,1]$ which are finite unions of dyadic intervals $I_s^m$ of length $1/2^m$. Let us also denote by $\mu$ the Lebesgue measure in $[0,1]$.

\begin{lemma}\label{aux:subsepacesR}
Let $g\in L^1([0,1])$, $g\geq 0$, $A\in D_m$ and $ \lambda \in (0,1)$. Then there exists an integer $n_\lambda >m$ such that for every $n\geq n_\lambda$ there are sets $A_0$, and $A_1\in D_n$ verifying 
\begin{enumerate}
\item $\mu (A_0)=\mu(A_1)=\frac 12 \mu(A)$;
\item $\frac{\lambda}{2} \int_A g \, d\mu \leq \int_{A_i} g \, d\mu\leq \frac{1}{2\lambda} \int_A g \, d\mu$, $i=0,1$;
\item  $r_{n| A_i}(t)=(-1)^i$,  $i=0,1$.
\end{enumerate}
\end{lemma}

\begin{proof}
It is not restrictive to assume that $\int_A g\,d\mu=1$. Now 
 we can find an integer $k>m$, a subset $S\subset\{0,1\}^k $, and a simple function $$ \varphi=\sum_{s\in S} a_s\chi_{I_s^k}$$ such that 
$A=\bigcup_{s\in S} I_s^k$ and 
\begin{equation}\label{eq:aproxsimple}\int_A |g-\varphi| \, d\mu < \delta=\frac{1-\lambda}{4}.\end{equation}  
Let $n_\lambda=k+1$,  and $n\geq n_\lambda$. We define  $A_0$ as the union of  dyadic intervals $I_d^n\subset A$ with $d=\{d_1,...,d_n\}$ and $d_n=0$,  and $A_1=A\setminus A_0$. Then,  $A=A_0\bigcup A_1$, $A_i\in D_n$  and $r_{n| A_i}(t)=(-1)^i$ for $i=0,1$. Having in mind that  for each $I_d^n\subset A$ the predecessor $s=\{d_1,...,d_k\}$ of $d$  is in $S$ and   $\varphi$ is constant on $I_d^n$, one can  deduce that
  $\int_{A_i}\varphi \, d\mu=\frac{1}{2}\int_A \varphi \, d\mu$. To finish the proof, we use \eqref{eq:aproxsimple} to get
$$\frac{\lambda}{2}=\frac{1}{2}-2\delta\leq\int_{A_i} \varphi\, d\mu-\delta\leq\int_{A_i} g\, d\mu \leq\int_{A_i} \varphi\, d\mu+\delta\leq \frac{1}{2}+2\delta\leq \frac{1}{2\lambda}.
$$
 \end{proof}

 We get the proof of Proposition \ref{subespacioR} iterating this last Lemma. Let $\eta\in (0,1)$ and choose a sequence $\lambda_k\in (0,1)$ such that $\prod_{k=1}^{+\infty}\lambda_k=\eta$.

We start by applying  Lemma \ref{aux:subsepacesR} for $\lambda_1\in(0,1)$ to the function $f$ and the set $A=[0,1]$ and we get $n_1$, $A_{\{0\}}^1, A_{\{1\}}^1\in D_{n_1}$ such that $$\lambda_1 \frac{1}{2}\leq \int_{A_{\{i\}}^1}f\, d\mu\leq \frac{1}{\lambda_1} \frac{1}{2},$$ and $r_{n_1|A_{\{i\}}^1}=(-1)^{i}$, $i=0,1$. 

Assume we have chosen $n_1<n_2<...<n_k$ and $A_{s}^p\in D_{n_p}$ with $s\in\{0,1\}^{p}$, $1\leq p\leq k$, in such a way that for $1<p\leq k$, $A_{\{s,0\}}^p$ and $A_{\{s,1\}}^p$ is the partition of $A_s^{p-1}$ that verifies the statement of Lemma \ref{aux:subsepacesR} for this set, the function $f$, $n=n_p$ and $\lambda=\lambda_p$.

To continue the construction by induction, for each $s\in\{0,1\}^k$ we apply   Lemma  \ref{aux:subsepacesR}        to the set $A_s^k$ , the function $f$, and the number $\lambda_{k+1}$, and we find a number $n(s,\lambda_{k+1})>n_k$ that satisfies its statement. We choose $n_{k+1}=\max\{n(s,\lambda_{k+1}): s\in\{0,1\}^k\}$, and we consider for each $s$, the partitions provided by the lemma  to obtain $A_{\{s,0\}}^{k+1}$, and $A_{\{s,1\}}^{k+1}\in D_{n_{k+1}}$
such that $A_s^k=A_{\{s,0\}}^{k+1}\bigcup A_{\{s,1\}}^{k+1}$,
$$\lambda_{k+1} \frac{1}{2} \, \int_{A_s^k}f\,d\mu\leq \int_{A_{\{s,i\}}^{k+1}}f\,d\mu\leq 
\frac{1}{\lambda_{k+1}} \frac{1}{2} \, \int_{A_s^k} f\, d\mu,$$
and $r_{n_{k+1}|A_{\{s,i\}}^{k+1}}=(-1)^i$, $i=0,1$. 

Under this construction, for any finite sequence of scalars $a_1,...,a_k$, the function
$\sum_{j=1}^k a_j r_j$ is constant on each $I_s^k$ and takes the same value as $\sum_{j=1}^k a_j r_{n_j}$ on $A_s^k$.

On the other hand,  the induction gives the inequalities:
$$\eta \mu(I_s^k)\leq \prod_{p=1}^{k}\lambda_p \frac{1}{2^{k}}
\leq \int_{A_{s}^{k}}f\,d\mu 
\leq \frac{1}{\prod_{p=1}^{k}\lambda_p} \frac{1}{2^{k}}\leq \frac{1}{\eta} \mu(I_s^k).$$

Now it is not difficult to prove that for any finite sequence of scalars $a_1,...,a_k$ we have
$$\eta \left\|\sum_{j=1}^{k} a_j r_j\right\|_1\leq  \left\|\sum_{j=1}^{k} a_j r_{n_j} f\right\|_1 	
\leq \frac{1}{\eta} \left\|\sum_{j=1}^{k} a_j r_j\right\|_1.$$
Thus, the subspace $E_f$ generated by the sequence $(f r_{n_k})_k$ in $L^1([0,1])$ is isomorphic to the subspace $R$ generated by the sequence of Rademacher functions $(r_k)_k$. 

\subsection*{An example of $f\in L^1([0,1]$ with ``{}bad''{} distribution.}

It is clear that  if $E_f$ is the reflexive subspace provided by Proposition \ref{subespacioR} for a positive function $f\in L^1([0,1])$ with $\|f\|_1=1$,  
then
$$G_{E_f}(t)\geq \int_t^{+\infty} F_f(s)ds.$$

\begin{prop} \label{funLog}
The function $f(x)=1/(x \log^2(x/e))$ defined in $[0,1]$ satisfies that $\int_0^1 f(x) dx=1$ and 
$$\int_t^{+\infty} F_f(s)ds 
\geq \frac{1}{4 \log(4t)}.$$
\end{prop}
\begin{proof}
Some simple calculus shows that $f$ decreases on $[0,1/e]$, increases on $[1/e,1]$, and $f(1)=1$, so for all $t>1$ we have
$F_{f}(t)=\lambda\{x: f(x)>t\}=f^{-1}(t),$
where $f^{-1}$ denotes the inverse function of $f$ defined on $[1,+\infty)$ with values in $(0,1/e)$.

Let us observe that $f(\sqrt{x e})=({4\sqrt{x}}/{\sqrt{e}})f(x)$.  We fix  $0< x_0< 1/e$ such that $f(x_0)=1$ and $t_0=f(x_0^2/e)>1$. If $t>t_0$ and $x_t=f^{-1}(t)$, we have $x_t< x_0^2/e$,  $\sqrt{x_t e} < x_0$, $f(\sqrt{x_t e})>1$ and
$t=(\sqrt{e}f(\sqrt{x_t e})/({4\sqrt{x_t}})>\sqrt{e}/({4\sqrt{x_t}}).$
And we have 
$x_t =f^{-1}(t)> \frac{e}{16 t^2}, \text{ for } t>t_0.$
From this inequality we get 
$$\int_{\{x: f(x)>t\}} f(x) dx =\int_{0}^{f^{-1}(t)} f(x) dx\geq\int_0^{\frac{e} {16 t^2}}f(x) dx=\frac{1} {2\log (4t)}, $$   
and
$$ t F_f(t)=t f^{-1}(t)=f(x_t) x_t= \frac{1}{\log^2(x_t/e)}< \frac{1} {\log^2(16 t^2)} \leq \frac{1}{4 \log(4t)}.$$

To finish, observe that for $t>t_0$ we have
$$
\int_t^{+\infty} F_f(s)ds =\int_{\{x: f(x)>t\}} f(x) dx  - tF_f(t)
\geq \frac{1}{4 \log(4t)}.$$
\end{proof}

\end{document}